\pdfoutput=1
\RequirePackage{ifpdf}
\ifpdf 
\documentclass[pdftex]{sigma}
\else
\documentclass{sigma}
\fi

\numberwithin{equation}{section}

\newtheorem{Theorem}{Theorem}[section]
\newtheorem{Lemma}[Theorem]{Lemma}
\newtheorem{Proposition}[Theorem]{Proposition}
 { \theoremstyle{definition}
\newtheorem{Definition}[Theorem]{Definition}
\newtheorem{Remark}[Theorem]{Remark} }

\begin{document}

\allowdisplaybreaks

\newcommand{\arXivNumber}{1704.04924}

\renewcommand{\PaperNumber}{072}

\FirstPageHeading

\ShortArticleName{On the Automorphisms of a Rank One Deligne--Hitchin Moduli Space}

\ArticleName{On the Automorphisms of a Rank One\\ Deligne--Hitchin Moduli Space}

\Author{Indranil BISWAS~$^\dag$ and Sebastian HELLER~$^\ddag$}

\AuthorNameForHeading{I.~Biswas and S.~Heller}

\Address{$^\dag$~School of Mathematics, Tata Institute of Fundamental Research,\\
\hphantom{$^\dag$}~Homi Bhabha Road, Mumbai 400005, India}
\EmailD{\href{mailto:indranil@math.tifr.res.in}{indranil@math.tifr.res.in}}

\Address{$^\ddag$~Institut f\"ur Differentialgeometrie, Universit\"at Hannover,\\
\hphantom{$^\ddag$}~Welfengarten 1, D-30167 Hannover, Germany}
\EmailD{\href{mailto:seb.heller@gmail.com}{seb.heller@gmail.com}}

\ArticleDates{Received May 13, 2017, in f\/inal form September 01, 2017; Published online September 06, 2017}

\Abstract{Let $X$ be a compact connected Riemann surface of genus $g \geq 2$, and let ${\mathcal M}_{\rm DH}$ be the rank one Deligne--Hitchin moduli space associated to $X$. It is known that ${\mathcal M}_{\rm DH}$ is the twistor space for the hyper-K\"ahler structure on the moduli space of rank one holomorphic connections on $X$. We investigate the group $\operatorname{Aut}({\mathcal M}_{\rm DH})$ of all holomorphic automorphisms of ${\mathcal M}_{\rm DH}$. The connected component of $\operatorname{Aut}({\mathcal M}_{\rm DH})$ containing the identity automorphism is computed. There is a natural element of $H^2({\mathcal M}_{\rm DH}, {\mathbb Z})$. We also compute the subgroup of~$\operatorname{Aut}({\mathcal M}_{\rm DH})$ that f\/ixes this second cohomology class. Since ${\mathcal M}_{\rm DH}$ admits an ample rational curve, the notion of algebraic dimension extends to it by a theorem of Verbitsky. We prove that ${\mathcal M}_{\rm DH}$ is Moishezon.}

\Keywords{Hodge moduli space; Deligne--Hitchin moduli space; $\lambda$-connections; Moishezon twistor space}

\Classification{14D20; 14J50; 14H60}

\section{Introduction}

The moduli spaces of Higgs bundles on a compact Riemann surface arise in various contexts and are extensively studied. One of the reasons for their usefulness is the nonabelian Hodge correspondence which identif\/ies the moduli space of semi-stable Higgs bundles of rank $r$ and degree zero on a compact Riemann surface $X$ with the character
variety \begin{gather*}\operatorname{Hom}(\pi_1(X), \text{GL}(r,{\mathbb C}))/\!\!/ \text{PGL}(r,{\mathbb C})\end{gather*}
(see~\cite{Co, Do, H, Si}). This identif\/ication is a $C^\infty$ dif\/feomorphism between the open dense subset consisting of stable Higgs bundles and the open dense subset consisting of irreducible representations. However, this dif\/feomorphism is not holomorphic. In other words, there are two dif\/ferent complex structures on a moduli space of Higgs bundles, one from being the moduli space of Higgs bundles and the other given by the complex structure of the character variety via the above mentioned identif\/ication. In fact, these two complex structures are a part of a~natural hyper-K\"ahler structure on a moduli spaces of Higgs bundles over~$X$~\cite{H}. It was noticed by Deligne and Hitchin that the twistor space associated to this hyper-K\"ahler manifold also has an interpretation as a moduli space associated to~$X$. These twistor spaces are known as the Deligne--Hitchin moduli space; see~\cite{Si95, Si08} for many properties of these spaces.

In a recent interesting paper \cite{Ba}, Baraglia computed the various automorphism groups associated to a moduli space of Higgs bundles on~$X$ (like holomorphic automorphisms preserving the holomorphic structure, holomorphic isometries, hypercomplex automorphisms, hyper-K\"ahler automorphisms et cetera). In~\cite{Ba}, the structure group is $\text{SL}(r, {\mathbb C})$, hence the moduli space parametrizes Higgs bundles of f\/ixed determinant with the trace of the Higgs f\/ield being zero. Inspired by~\cite{Ba}, in \cite{BBS} the case of ${\mathbb C}^*$-bundles was investigated, where ${\mathbb C}^*$ is the multiplicative group of nonzero complex numbers. More precisely, the automorphism groups of the moduli spaces of rank one Higgs bundles, rank one holomorphic connections and ${\mathbb C}^*$-character varieties were studied.

Our main aim here is to investigate the automorphisms of the rank one Deligne--Hitchin moduli space. The Deligne--Hitchin moduli spaces are not algebraic varieties; they are complex manifolds. However, they are built out of Hodge moduli spaces which are complex quasiprojective varieties. The rank one Hodge moduli space ${\mathcal M}_{\rm Hod} = {\mathcal M}^X_{\rm Hod}(1)$ parametrizes all rank one $\lambda$-connections on a compact Riemann surface $X$ with $\lambda$ running over~$\mathbb C$. The rank one Deligne--Hitchin moduli space ${\mathcal M}_{\rm DH} = {\mathcal M}_{\rm DH}(X)$ is constructed by gluing ${\mathcal M}^X_{\rm Hod}(1)$ and ${\mathcal M}^{\overline X}_{\rm Hod}(1)$ over the inverse image of ${\mathbb C}^* \subset {\mathbb C}$, where ${\overline X}$ is the Riemann surface conjugate to~$X$.

It is natural to ask for the automorphism group of the rank one Deligne--Hitchin moduli space as it is the only case missing for rank one, see~\cite{BBS}. In general, automorphisms of moduli spaces associated to a Riemann surface play an important role in mirror symmetry and help to improve our understanding of these objects. The rank one case serves as a motivating and more simple example, where we can gain intuition for what should be expected in general. On the other hand, this task is also interesting for itself as we cannot use an algebraic structure of the moduli space, and algebraicity was at the heart of the proofs in \cite{BBS}. As a nice additional observation we would like to mention that the automorphisms of the Deligne--Hitchin moduli space are algebraic when restricted to some of its subspaces, i.e., to the Dolbeault, the de Rham and the Hodge moduli space. Finally, and most importantly, we hope that our techniques can be generalized to study automorphism groups of the Deligne--Hitchin moduli space for higher rank. This is of particular interest as certain sections thereof correspond to solutions to important geometric PDE's such as harmonic maps including Hitchin's self-duality equations, and certain automorphisms might give rise to transformations between some classes of PDE's.

We now describe the results proved here. Let $\operatorname{Aut}({\mathcal M}_{\rm Hod})$ be the group of all algebraic automorphisms of ${\mathcal M}_{\rm Hod}$, and let $\operatorname{Aut}({\mathcal M}_{\rm Hod})_0$ be the connected component of it containing the identity automorphism. The holomorphic cotangent bundle of $X$ will be denoted by~$K_X$.

\begin{Theorem}\label{ti1} The group $\operatorname{Aut}({\mathcal M}_{\rm Hod})_0$ fits in a short exact sequence of groups
\begin{gather*}
0 \longrightarrow {\mathbb V} \stackrel{\iota}{\longrightarrow} \operatorname{Aut}({\mathcal M}_{\rm Hod})_0 \stackrel{h}{\longrightarrow}
\operatorname{Pic}^0(X)\times {\mathbb C}^* \longrightarrow 0 ,
\end{gather*}
where ${\mathbb V}$ is the space of all algebraic maps from $\mathbb C$ to the affine space $H^0(X, K_X)$.
\end{Theorem}

Def\/ine $\Gamma_X := \operatorname{Aut}(X)$ if $X$ is hyperelliptic, and $\Gamma_X = \operatorname{Aut}(X)\oplus ({\mathbb Z}/2{\mathbb Z})$ if $X$ is non-hyperelliptic. This group $\Gamma_X$ has a natural action on ${\mathcal M}_{\rm Hod}$ via holomorphic automorphisms; see Section~\ref{Autocoh}. Let
\begin{gather*}
\widetilde{\Gamma}_X := \operatorname{Aut}({\mathcal M}_{\rm Hod})_0\rtimes \Gamma_X
\end{gather*}
be the corresponding semi-direct product (see \eqref{tmu2}). There is a natural cohomology class
\begin{gather*}
\theta \in H^2({\mathcal M}_{\rm Hod}, {\mathbb Z}) ,
\end{gather*}
and the action of $\widetilde{\Gamma}_X$ on ${\mathcal M}_{\rm Hod}$ f\/ixes~$\theta$.

\begin{Proposition} Assume that ${\rm genus}(X) > 1$. The group $\widetilde{\Gamma}_X$ coincides with the subgroup of $\operatorname{Aut}({\mathcal M}_{\rm Hod})$ that fixes $\theta$.
\end{Proposition}

Let ${\mathcal M}_{\rm DH} = {\mathcal M}_{\rm DH}(X)$ be the rank one Deligne--Hitchin moduli space for a compact Riemann surface $X$. Let $\operatorname{Aut}({\mathcal M}_{\rm DH})$ be the group of all holomorphic automorphisms of ${\mathcal M}_{\rm DH}$, and let $\operatorname{Aut}({\mathcal M}_{\rm DH})_0 \subset \operatorname{Aut}({\mathcal M}_{\rm DH})$ be the connected component of it containing the identity map. The moduli space of rank one connections on $X$ will be denoted by ${\mathcal M}_{dR}(X)$; it is an algebraic group.

\begin{Theorem}\label{ti2}
There is an isomorphism
\begin{gather*}\operatorname{Aut}({\mathcal M}_{\rm DH})_0 = \mathcal M_{dR}(X)\times \big(\big(H^0(X, K_X)\times\overline{H^0(X, K_X)}\big)\rtimes \mathbb C^*\big) .\end{gather*}
\end{Theorem}

The group structure of $\mathcal M_{dR}(X)\times \big(\big(H^0(X, K_X)\times\overline{H^0(X, K_X)}\big)\rtimes \mathbb C^*\big)$ in Theorem~\ref{ti2} is described in Theorem~\ref{thmDHa}.

The above group $\Gamma_X$ acts on ${\mathcal M}_{\rm DH}$. Let
\begin{gather*}
\widetilde{\Gamma}_{{\rm DH},X} := \operatorname{Aut}({\mathcal M}_{\rm DH})_0\rtimes \Gamma_X
\end{gather*}
be the corresponding semi-direct product. There is a natural cohomology class
\begin{gather*}
\theta_X \in H^2({\mathcal M}_{\rm DH}, {\mathbb Z}) ,
\end{gather*}
and the action of $\widetilde{\Gamma}_{{\rm DH},X}$ on ${\mathcal M}_{\rm DH}$ f\/ixes~$\theta_X$.

\begin{Theorem}Assume that ${\rm genus}(X) > 1$. The group $\widetilde{\Gamma}_{{\rm DH},X}$ coincides with the subgroup of~$\operatorname{Aut}({\mathcal M}_{\rm DH})$ that fixes~$\theta_X$.
\end{Theorem}

Although ${\mathcal M}_{\rm DH}$ is noncompact, being a twistor space it contains ample rational curves. Verbitsky has shown that in such a context the notion of algebraic dimension continues to hold~\cite{Ve}.

We prove the following:

\begin{Proposition} The algebraic dimension of ${\mathcal M}_{\rm DH}(X)$ coincides with $\dim {\mathcal M}_{\rm DH}(X) = 1+2\cdot {\rm genus}(X)$, or in other words, ${\mathcal M}_{\rm DH}(X)$ is Moishezon\footnote{Here we are using the terminology of~\cite{Ve}; one of the referees pointed out that a more correct terminology would be \textit{Moishezon--Verbitsky manifold}.}.
\end{Proposition}

\section{Automorphisms of the rank one Hodge moduli space}\label{auto-hodge}

\subsection{Connected component of the automorphism group}

Let $X$ be an irreducible smooth complex projective curve of genus $g$, with $g \geq 1$. The holomorphic cotangent bundle of $X$ will be denoted by $K_X$.
\begin{Definition} For any $\lambda \in \mathbb C$, a rank one $\lambda$-connection on $X$ is a pair of the form $(L, D)$, where $L$ is a holomorphic line bundle on~$X$ and
\begin{gather*}
D \colon \ L \longrightarrow L\otimes K_X
\end{gather*}
is a holomorphic dif\/ferential operator of order no more than one such that
\begin{gather*}
D(f_0s) = f_0D(s)+\lambda\cdot s\otimes (df_0)
\end{gather*}
for every locally def\/ined holomorphic section $s$ of $L$ and every locally def\/ined holomorphic function $f_0$ on $X$.
\end{Definition}

So, if $(L, D)$ is a $\lambda$-connection with $\lambda \not= 0$, then $D/\lambda$ is a holomorphic connection on $L$; a~$0$-connection $D$ on $L$ is an element of $H^0(X, K_X)$ because such a $D$ is ${\mathcal O}_X$-linear. Note that for a~$\lambda$-connection $(L, D)$, with $\lambda \not= 0$ we have $\text{degree}(L) = 0$ because $L$ admits a holomorphic connection. Also, note that any holomorphic connection on a Riemann surface is automatically f\/lat.

If $(L, D)$ is a $0$-connection we will always impose the condition that $\text{degree}(L) = 0$.

Let
\begin{gather*}
{\mathcal M}_{\rm Hod} = {\mathcal M}^X_{\rm Hod}(1)
\end{gather*}
be the Hodge moduli space of rank one $\lambda$-connections. So, the points of ${\mathcal M}_{\rm Hod}$ parametrize triples of the form $(\lambda, L, D)$, where $\lambda \in \mathbb C$, $L$ is a holomorphic line bundle on $X$ of degree zero and~$D$ is a $\lambda$-connection on~$L$.

Def\/ine an algebraic morphism
\begin{gather}\label{lf}
f \colon \ {\mathcal M}_{\rm Hod} \longrightarrow {\mathbb C} , \qquad (\lambda, L, D) \longmapsto \lambda .
\end{gather}
For any $\lambda \in \mathbb C$, let
\begin{gather*}
f^\lambda := f^{-1}(\lambda) \subset {\mathcal M}_{\rm Hod}
\end{gather*}
be the f\/iber over $\lambda$. By def\/inition, the f\/iber $f^0$ is the moduli space of Higgs line bundles of degree zero $\operatorname{Pic}^0(X)\times H^0(X, K_X)$ on $X$, and the f\/iber $f^1$ is the moduli space of rank one holomorphic connections on $X$. The latter space is also called the de Rham moduli space. For $\lambda \not= 0$, the f\/iber $f^\lambda$ is canonically identif\/ied with $f^1$ by the map $(L, D) \longmapsto (L, D/\lambda)$. We note that $f^1$ is biholomorphic to the Betti moduli space $\operatorname{Hom}(\pi_1(X), {\mathbb C}^*) = ({\mathbb C}^*)^{2g}$ by sending a holomorphic connection to its monodromy representation, and hence $f^1$ does not admit a~compact submanifold of positive dimension. Since $\operatorname{Pic}^0(X)\times H^0(X, K_X)$ has the projective variety $\operatorname{Pic}^0(X)$ of positive dimension as a~subvariety, it follows that~$f^0$ is not biholomorphic to~$f^1$.

The group of all algebraic automorphisms of the quasiprojective variety
${\mathcal M}_{\rm Hod}$ will be denoted by $\operatorname{Aut}({\mathcal M}_{\rm Hod})$. Let
\begin{gather*}
\operatorname{Aut}({\mathcal M}_{\rm Hod})_0 \subset \operatorname{Aut}({\mathcal M}_{\rm Hod})
\end{gather*}
be the connected component containing the identity automorphism and
 let
\begin{gather*}
{\mathbb V} := \operatorname{Morphisms}\big({\mathbb C}, H^0(X, K_X)\big)
\end{gather*}
be the inf\/inite dimensional complex vector space parametrizing all algebraic maps from $\mathbb C$ to the af\/f\/ine space $H^0(X, K_X)$. Then we have the following theorem.
\begin{Theorem}\label{thm1}
The group $\operatorname{Aut}({\mathcal M}_{\rm Hod})_0$ fits in a short exact sequence of groups
\begin{gather*}
0 \longrightarrow {\mathbb V} \stackrel{\iota}{\longrightarrow} \operatorname{Aut}({\mathcal M}_{\rm Hod})_0 \stackrel{h}{\longrightarrow}
\operatorname{Pic}^0(X)\times {\mathbb C}^* \longrightarrow 0 .
\end{gather*}
\end{Theorem}

\begin{proof} Let
\begin{gather*}
T \colon \ {\mathcal M}_{\rm Hod} \longrightarrow {\mathcal M}_{\rm Hod}
\end{gather*}
be any algebraic automorphism that lies in $\operatorname{Aut}({\mathcal M}_{\rm Hod})_0$. It is known that the f\/iber $f^1$ does not admit any nonconstant algebraic function \cite[Proposition~2.2]{BBS}. Since $f^\lambda$ is isomorphic to $f^1$ for every $\lambda \in {\mathbb C}^*$, it follows that $f^\lambda$ does not admit any nonconstant algebraic function if $\lambda \in {\mathbb C}^*$. Hence the composition
\begin{gather}\label{e1}
f^\lambda \stackrel{T\vert_{f^\lambda}}{\longrightarrow} {\mathcal M}_{\rm Hod} \stackrel{f}{\longrightarrow} \mathbb C
\end{gather}
is a constant function for all $\lambda \in {\mathbb C}^*$. Note that the point $f\circ T(f^{\lambda})$ lies in ${\mathbb C}^*$ because $f^0$ is not isomorphic to $f^\lambda$. Since this also holds for~$T^{-1}$, it follows that the composition in~\eqref{e1} is a~constant function for each $\lambda \in {\mathbb C}$. This implies that there is an algebraic automorphism
\begin{gather*}
\tau_0 \colon \ {\mathbb C} \longrightarrow \mathbb C
\end{gather*}
such that
\begin{gather*}
f\circ T = \tau_0 \circ f .
\end{gather*}
Since $f^0$ is not isomorphic to any $f^\lambda$, $\lambda \in {\mathbb C}^*$, it follows that $\tau_0(0) = 0$. Therefore, $\tau_0$ coincides with the multiplication of $\mathbb C$ by a f\/ixed nonzero number; let
\begin{gather}\label{e4}
\tau \in {\mathbb C}^*
\end{gather}
be such that $\tau_0(z) = \tau \cdot z$ for all $z \in \mathbb C$.

The group of all automorphisms of the variety $\operatorname{Pic}^0(X)$ will be denoted by $\operatorname{Aut}(\operatorname{Pic}^0(X))$ (since the variety $\operatorname{Pic}^0(X)$ is projective, any holomorphic automorphism of it is algebraic). The connected component of~$\operatorname{Aut}(\operatorname{Pic}^0(X))$ containing the identity automorphism will be denoted by~$\operatorname{Aut}(\operatorname{Pic}^0(X))_0$. The automorphisms of $\operatorname{Pic}^0(X)$ given by translations of the group $\operatorname{Pic}^0(X)$ lie in $\operatorname{Aut}(\operatorname{Pic}^0(X))_0$. In fact, this way $\operatorname{Pic}^0(X)$ gets identif\/ied with
$\operatorname{Aut}(\operatorname{Pic}^0(X))_0$; note that the Lie algebra $H^0(\operatorname{Pic}^0(X), T\operatorname{Pic}^0(X))$ of~$\operatorname{Aut}(\operatorname{Pic}^0(X))_0$ coincides with the Lie algebra of $\operatorname{Pic}^0(X)$ by evaluating sections of~$T\operatorname{Pic}^0(X)$ at the identity element.

For any $\lambda \in \mathbb C$, let $\phi^\lambda$ be the morphism def\/ined by
\begin{gather}\label{e7}
\phi^\lambda \colon \ f^\lambda \longrightarrow \operatorname{Pic}^0(X) ,\qquad (L, D) \longmapsto L .
\end{gather}
The f\/ibers of $\phi^\lambda$ are isomorphic to $H^0(X, K_X)$, because the space of all holomorphic connections on a line bundle $L \in \operatorname{Pic}^0(X)$ is an af\/f\/ine space for the vector space $H^0(X, K_X)$, and $H^0(X, K_X)$ is the space of all Higgs f\/ields on any line bundle. There is no nonconstant algebraic map from an af\/f\/ine space to an abelian variety. Hence, there is no nonconstant algebraic map from a f\/iber of $\phi^\lambda$ to $\operatorname{Pic}^0(X)$. Consequently, we get a map
\begin{gather*}
\Phi \colon \  {\mathbb C} \longrightarrow \operatorname{Aut}\big(\operatorname{Pic}^0(X)\big),
\end{gather*}
which is uniquely determined by the condition that the following diagram is commutative
\begin{gather*}
\begin{matrix}
f^\lambda & \stackrel{T\vert_{f^\lambda}}{\longrightarrow} & f^{\tau\lambda}\\
~ ~~ \Big\downarrow \phi^\lambda && ~ ~~~~ \Big\downarrow \phi^{\tau\lambda}\\
\operatorname{Pic}^0(X) & \stackrel{\Phi(\lambda)}{\longrightarrow} &\operatorname{Pic}^0(X)
\end{matrix}
\end{gather*}
for all $\lambda \in \mathbb C$, where $\tau$ is the complex number in \eqref{e4}$ $. Note that since $T \in \operatorname{Aut}({\mathcal M}_{\rm Hod})_0$, it follows that the image of $\Phi$ lies in $\operatorname{Aut}\big(\operatorname{Pic}^0(X)\big)_0 = \operatorname{Pic}^0(X) \subset \operatorname{Aut}\big(\operatorname{Pic}^0(X)\big)$. Thus, $\Phi$ is a~constant map, again because there is no nonconstant algebraic map from $\mathbb C$ to $\operatorname{Pic}^0(X)$. Let{\samepage
\begin{gather}\label{e5}
\widehat{\Phi} = \Phi({\mathbb C}) \in \operatorname{Pic}^0(X)
\end{gather}
be the image of the map $\Phi$.}

Let
\begin{gather*}
h \colon \ \operatorname{Aut}({\mathcal M}_{\rm Hod})_0 \longrightarrow \operatorname{Pic}^0(X)\times {\mathbb C}^*
\end{gather*}
be the map def\/ined by $T \longmapsto \big(\widehat{\Phi}, \tau\big)$. It is straight-forward to check that $h$ is a group homomorphism.

We will now show that $h$ is surjective. For this, f\/irst note that the multiplicative action of~${\mathbb C}^*$ on $\mathbb C$ has a natural lift to an action of ${\mathbb C}^*$ on ${\mathcal M}_{\rm Hod}$. Indeed, any $c \in {\mathbb C}^*$ acts on ${\mathcal M}_{\rm Hod}$ as
\begin{gather*}
(\lambda, L, D) \longmapsto (c\cdot\lambda, L, c\cdot D) .
\end{gather*}
For all these automorphisms of ${\mathcal M}_{\rm Hod}$ given by the action of ${\mathbb C}^*$, the corresponding elements of~$\operatorname{Pic}^0(X)$ (see \eqref{e5}) coincide with the identity element. Therefore, to prove that~$h$ is surjective, it suf\/f\/ices to show that the composition of $h$ with the projection $\operatorname{Pic}^0(X)\times {\mathbb C}^*  \longrightarrow \operatorname{Pic}^0(X)$ is surjective. Take any $L_0 \in \operatorname{Pic}^0(X)$ and f\/ix a holomorphic connection~$D_0$ on~$L_0$. Def\/ine
\begin{gather*}
\beta \colon \  {\mathcal M}_{\rm Hod} \longrightarrow {\mathcal M}_{\rm Hod} ,\qquad
(\lambda, L, D) \longmapsto (\lambda, L\otimes L_0, (D\otimes \text{Id}_{L_0}) +(\lambda\cdot \text{Id}_L\otimes D_0)) .
\end{gather*}
It is straight-forward to check that $\beta \in \operatorname{Aut}({\mathcal M}_{\rm Hod})$; its inverse is the corresponding map for $(L_0^*, D_0^*)$. Since the moduli space of rank one connections on $X$ is connected, it follows that $\beta \in \operatorname{Aut}({\mathcal M}_{\rm Hod})_0$; note that $\beta$ for the trivial connection $({\mathcal O}_X, d)$ is the identity map of ${\mathcal M}_{\rm Hod}$. The element in $\operatorname{Pic}^0(X)$ corresponding to $\beta$ (see~\eqref{e5}) is clearly $L_0$. Therefore, the composition of $h$ with the projection $\operatorname{Pic}^0(X)\times {\mathbb C}^*  \longrightarrow \operatorname{Pic}^0(X)$ and hence $h$ itself are surjective.

Next, we need to def\/ine ${\mathbb V} \stackrel{\iota}{\longrightarrow} \operatorname{Aut}({\mathcal M}_{\rm Hod})_0$. To do so, consider for $v \in \mathbb V$ the automorphism
\begin{gather}\label{e6}
\iota_v \colon \ {\mathcal M}_{\rm Hod} \longrightarrow {\mathcal M}_{\rm Hod} ,\qquad (\lambda, L, D) \longmapsto (\lambda, L, D+v(\lambda)) .
\end{gather}
Clearly, we have $\iota_v \in \operatorname{Aut}({\mathcal M}_{\rm Hod})_0$, and also $\iota_v \in \operatorname{kernel}(h)$. Therefore, there is an injective homomorphism
\begin{gather*}
\iota \colon \ {\mathbb V} \longrightarrow \operatorname{kernel}(h) \subset \operatorname{Aut}({\mathcal M}_{\rm Hod})_0 ,\qquad v \longmapsto \iota_v .
\end{gather*}
We will prove that $\operatorname{image}(\iota) = \operatorname{kernel}(h)$.

In order to prove that $\operatorname{image}(\iota) = \operatorname{kernel}(h)$, take any $T \in \operatorname{Aut}({\mathcal M}_{\rm Hod})_0$ such that
\begin{gather*}T \in \operatorname{kernel}(h) .\end{gather*} We have $T(f^\lambda) = f^\lambda$ for all
$\lambda \in {\mathbb C}^*$ because the constant in \eqref{e4} for $T$ is~$1$. Since the element in \eqref{e5} corresponding to $T$ is the trivial line bundle, we get a morphism
\begin{gather*}
\delta_\lambda \colon \  f^\lambda \longrightarrow H^0(X, K_X) ,\qquad z \longmapsto T(z) - z ;
\end{gather*}
note that since the f\/ibers of $\phi^\lambda$ (constructed in \eqref{e7}) are af\/f\/ine spaces for $H^0(X, K_X)$, and $\phi^\lambda(T(z)) = \phi^\lambda(z)$ for all $z \in f^\lambda$, we have $T(z) - z \in H^0(X, K_X)$. As there are no nonconstant algebraic functions on $f^\lambda$ \cite[Proposition~2.2]{BBS}, it follows that the above function $\delta_\lambda$ is a constant one.

All automorphism of the af\/f\/ine space $H^0(X, K_X)$ are of the form $u \longmapsto A(u)+u_0$, where $A \in \text{GL}(H^0(X, K_X))$ and $u_0 \in H^0(X, K_X)$. Considering the restrictions of $T$ to open subsets of the form $h^{-1}(V\times U)$, where $U$ is an analytic neighborhood of $0 \in \mathbb C$ and $V$ is an analytic open subset of $\operatorname{Pic}^0(X)$, it now follows that the above map
\begin{gather*}\lambda \longmapsto \operatorname{image}(\delta_\lambda) \in H^0(X, K_X)\end{gather*}
extends across $0 \in \mathbb C$ as a holomorphic map from $\mathbb C$ to $H^0(X, K_X)$. In other words, there is a~unique element $v \in {\mathbb V}$ such that $v(\lambda) = \operatorname{image} (\delta_\lambda)$ for all $\lambda \in {\mathbb C}^*$. Clearly, we have $\iota_v = T$, where~$\iota_v$ is constructed in~\eqref{e6}. This implies that $\operatorname{image}(\iota) = \operatorname{kernel}(h)$. This completes the proof of the theorem.
\end{proof}

\subsection{Automorphisms preserving a cohomology class}\label{Autocoh}

In this subsection we assume that $g = \operatorname{genus}(X) > 1$.

Since any holomorphic line bundle of degree zero on the compact Riemann surface $X$ admits a unique f\/lat connection with unitary monodromy, the Picard group $\operatorname{Pic}^0(X)$ is identif\/ied with
\begin{gather*}\operatorname{Hom}(\pi_1(X), \text{U}(1)) = \operatorname{Hom}(H_1(X, {\mathbb Z}), \text{U}(1)) .\end{gather*}
There is a natural symplectic form
\begin{gather}\label{tt}
\widetilde\theta
\end{gather}
on the character variety $\operatorname{Hom}(\pi_1(X), \text{U}(1))$ \cite{AB,Go}. The cohomology class def\/ined by $\widetilde\theta$ is integral. Let
\begin{gather*}
\theta_0 \in H^2\big(\operatorname{Pic}^0(X), {\mathbb Z}\big)
\end{gather*}
be the cohomology class given by $\widetilde\theta$ in \eqref{tt}. This $\theta_0$ is a principal polarization on $\operatorname{Pic}^0(X)$. More precisely, it is the class of a theta divisor on $\operatorname{Pic}^0(X)$. Consider the projection
\begin{gather}\label{phi}
\phi \colon \ {\mathcal M}_{\rm Hod} \longrightarrow \operatorname{Pic}^0(X) ,\qquad (\lambda, L, D) \longmapsto L ;
\end{gather}
so, the restriction of $\phi$ to $f^\lambda$ is the map $\phi^\lambda$ in~\eqref{e7}. Def\/ine the cohomology class
\begin{gather}\label{the}
\theta := \phi^*\theta_0 = [\phi^*\widetilde\theta] \in H^2({\mathcal M}_{\rm Hod}, {\mathbb Z}) .
\end{gather}

The moduli space $f^1$ of holomorphic rank one connections
(see \eqref{e1}) has a natural holomorphic symplectic form $\theta_h$.
The restriction of $\theta_h$ to the moduli space of holomorphic connections
$\operatorname{Hom}(\pi_1(X), \text{U}(1))$ with unitary monodromy coincides with the above
symplectic form $\widetilde\theta$. On the other hand $f^1$ has a deformation
retraction onto $\operatorname{Hom}(\pi_1(X), \text{U}(1))$; for example such
a deformation retraction is given by a deformation retraction of
${\mathbb C}^*$ to $\text{U}(1)$. Therefore, we conclude that the cohomology class of
$\theta_h$ coincides with the restriction of $\theta$ to $f^1$.

Let
\begin{gather}\label{autc}
\operatorname{Aut}^\theta({\mathcal M}_{\rm Hod}) \subset \operatorname{Aut}({\mathcal M}_{\rm Hod})
\end{gather}
be the subgroup consisting of all algebraic automorphisms of ${\mathcal M}_{\rm Hod}$
that f\/ixes the cohomology class $\theta$ in \eqref{the}.

Let $\operatorname{Aut}(X)$ denote the group of all holomorphic automorphisms of $X$. We have a
homomorphism
\begin{gather*}
\mu \colon \ \operatorname{Aut}(X)\oplus ({\mathbb Z}/2{\mathbb Z}) \longrightarrow
\operatorname{Aut}\big(\operatorname{Pic}^0(X)\big) , \qquad \mu(t,0)(L) = t^*L , \qquad \mu(t,1)(L) = t^*L^* .
\end{gather*}
We note that the restriction $\mu\vert_{\operatorname{Aut}(X)}$ is injective (recall that $g \geq 2$), and~$\mu$ is injective if and only if $X$ is non-hyperelliptic. If~$X$ is hyperelliptic, then the hyperelliptic involution acts on~$\operatorname{Pic}^0(X)$ as~$L \longmapsto L^*$. Def\/ine
\begin{gather}\label{GX}
\Gamma_X = \operatorname{image}(\mu) .
\end{gather}
So, $\Gamma_X = \operatorname{Aut}(X)$ if $X$ is hyperelliptic, and $\Gamma_X = \operatorname{Aut}(X)\oplus ({\mathbb Z}/2{\mathbb Z})$ if $X$ is non-hyperelliptic.

We also have a homomorphism
\begin{gather}\label{tmu}
\widetilde{\mu} \colon \  \operatorname{Aut}(X)\oplus ({\mathbb Z}/2{\mathbb Z}) \longrightarrow \operatorname{Aut}({\mathcal M}_{\rm Hod})
\end{gather}
def\/ined by $\widetilde{\mu}(t,0)(\lambda, L, D) = (\lambda, t^*L, t^*D)$ and $\widetilde{\mu}(t,1)(\lambda, L, D) = (\lambda, t^*L^*, t^*D^*)$. Note that if $\lambda = 0$, then $D^* = -D \in H^0(X, K_X)$. The homomorphism $\widetilde{\mu}$ in~\eqref{tmu} factors through the quotient~$\Gamma_X$ of $\operatorname{Aut}(X)\oplus ({\mathbb Z}/2{\mathbb Z})$ in \eqref{GX}.

Consider the action of $\Gamma_X$ on $\operatorname{Aut}({\mathcal M}_{\rm Hod})$ given by the composition of $\widetilde{\mu}$ with the adjoint action of~$\operatorname{Aut}({\mathcal M}_{\rm Hod})$ on itself. This action preserves the connected component $\operatorname{Aut}({\mathcal M}_{\rm Hod})_0$. Def\/ine the semi-direct product
\begin{gather}\label{tmu2}
\widetilde{\Gamma}_X := \operatorname{Aut}({\mathcal M}_{\rm Hod})_0\rtimes \Gamma_X
\end{gather}
for the above action of $\Gamma_X$ on $\operatorname{Aut}({\mathcal M}_{\rm Hod})_0$.

Consider the action of $\Gamma_X$ on ${\mathcal M}_{\rm Hod}$ given by $\widetilde{\mu}$ in \eqref{tmu}. This action clearly f\/ixes the cohomology class $\theta$ in \eqref{the}. The action of $\operatorname{Aut}({\mathcal M}_{\rm Hod})_0$ on ${\mathcal M}_{\rm Hod}$ f\/ixes $\theta$ because $\operatorname{Aut}({\mathcal M}_{\rm Hod})_0$ is connected. Also, $\operatorname{Aut}({\mathcal M}_{\rm Hod})_0$ is a normal subgroup of~$\operatorname{Aut}({\mathcal M}_{\rm Hod})$. Therefore, we get a~homomorphism
\begin{gather}\label{rho}
\rho \colon \ \widetilde{\Gamma}_X \longrightarrow \operatorname{Aut}^\theta({\mathcal M}_{\rm Hod}),
\end{gather}
where $\operatorname{Aut}^\theta({\mathcal M}_{\rm Hod})$ and $\widetilde{\Gamma}_X$ are constructed in~\eqref{autc} and~\eqref{tmu2} respectively.

\begin{Proposition}\label{propo1}
The homomorphism $\rho$ in \eqref{rho} is an isomorphism.
\end{Proposition}

\begin{proof}
We will f\/irst prove that $\rho$ is injective. For this note that $\phi$ in~\eqref{phi} induces an isomorphism of f\/irst cohomologies with coef\/f\/icients in $\mathbb C$. Indeed, the f\/ibers of $\phi$ are dif\/feomorphic to ${\mathbb C}\times H^0(X, K_X)$, hence $\phi$ induces an isomorphism $H^1\big(\operatorname{Pic}^0(X), {\mathbb C}\big) \stackrel{\phi^*}{\longrightarrow} H^1({\mathcal M}_{\rm Hod}, {\mathbb C})$. Take any
\begin{gather*}\gamma \in \operatorname{kernel}(\rho) \subset \widetilde{\Gamma}_X .\end{gather*}
Let $\gamma' \in \Gamma_X$ be the image of $\gamma$ by the natural projection of $\widetilde{\Gamma}_X$ to $\Gamma_X$. Let $\gamma'' \in \widetilde{\Gamma}_X$ be the image of $\gamma'$ by the natural inclusion $\Gamma_X \hookrightarrow  \widetilde{\Gamma}_X$. So, $\gamma$ and $\gamma''$ dif\/fer by an element of $\operatorname{Aut}({\mathcal M}_{\rm Hod})_0$. Since $\gamma$ acts trivially on~$H^1({\mathcal M}_{\rm Hod}, {\mathbb C})$ (as it acts trivially on~${\mathcal M}_{\rm Hod}$), and $\operatorname{Aut}({\mathcal M}_{\rm Hod})_0$ acts trivially on $H^1({\mathcal M}_{\rm Hod}, {\mathbb C})$ as it is connected, we conclude that $\gamma''$ acts trivially on $H^1({\mathcal M}_{\rm Hod}, {\mathbb C})$. This implies that $\gamma'$ acts trivially on $H^1(\operatorname{Pic}^0(X), {\mathbb C}) = H^1({\mathcal M}_{\rm Hod}, {\mathbb C})$. But, as noted earlier, this implies that $\gamma' = 1$. Hence, $\gamma \in \operatorname{Aut}({\mathcal M}_{\rm Hod})_0$. Now we conclude that $\gamma = 1$ because $\rho(\gamma)$ is the trivial automorphism of ${\mathcal M}_{\rm Hod}$. Hence $\rho$ is injective.

Let $T\in\operatorname{Aut}^\theta({\mathcal M}_{\rm Hod})$. The image
\begin{gather*}T(\{0\}\times \operatorname{Pic}^0(X)\times\{0\})\subset\mathcal M_{\rm Hod}\end{gather*}
is of the form
\begin{gather*}\{0\}\times \operatorname{Pic}^0(X)\times\{\omega\}\end{gather*}
for some $\omega\in H^0(X;K_X)$, as can be deduced analogously to the proof of Theorem~\ref{thm1}. By applying a suitable automorphism $T_0\in\operatorname{Aut}({\mathcal M}_{\rm Hod})_0$ we get
\begin{gather*}T_0\circ T(\{0\}\times \operatorname{Pic}^0(X)\times\{0\})=\{0\}\times \operatorname{Pic}^0(X)\times\{0\}.\end{gather*}
The group of automorphisms of the abelian variety $\operatorname{Pic}^0(X)$ that f\/ix the principal polarization~$\theta_0$ is generated by $\Gamma_X$ and the translations of $\operatorname{Pic}^0(X)$ (see \cite[p.~35, Hauptsatz]{We}). Because of the discussion in the beginning of Section~\ref{Autocoh} it follows that we can apply an automorphism $T_1\in\Gamma_X$ such that
\begin{gather*}T_1\circ T_0\circ T\end{gather*} is the identity whence restricted to $\{0\}\times \operatorname{Pic}^0(X)\times\{0\}$.
Clearly, this implies f\/irst that $T_1\circ T_0\circ T$ restricted to $f^0$ is homotopic to the identity and f\/inally that $T_1\circ T_0\circ T$ is homotopic to the identity, which f\/inishes the proof.
\end{proof}

\section[Holomorphic automorphisms of the Deligne--Hitchin moduli space]{Holomorphic automorphisms of the Deligne--Hitchin\\ moduli space}

\subsection{Connected component of the holomorphic automorphism group}\label{aut0mdh}

The smooth locus of the moduli space of f\/lat connections on a compact Riemann surface is equipped with a natural hyper-K\"ahler structure; see~\cite{H} for the case of ${\rm SL}(2,\mathbb C)$-connections, and \cite{GX} for a detailed treatment of the case of f\/lat line bundles.

A hyper-K\"ahler manifold is a Riemannian manifold $(M, g)$ which is K\"ahler for three complex structures $I$, $J$ and $K$ satisfying the quaternionic relations
\begin{gather*}IJ = -JI = K .\end{gather*} The interplay between the Riemannian, symplectic and complex geometric aspects of hyper-K\"ahler manifolds makes them fascinating mathematical objects. A hyper-K\"ahler mani\-fold~$M$ admits a twistor space which contains all of the information about $M$ in a~complex geometric fashion: The compatible complex structures on $M$ are parametrized by $\mathbb CP^1 = S^2 = \big\{(x_1, x_2, x_3) \in {\mathbb R}^3 \,|\, x^2_1+x^2_2+x^2_3 = 1\big\}$ via
\begin{gather*}(x_1, x_2, x_3) \longmapsto I_{(x,_1,x_2,x_3)} = x_1 I+x_2 J+x_3 K ,\end{gather*}
where $I$, $J$, $K$ are the three original complex structures on~$M$. The twistor space $\mathcal Z$ is a~complex manifold with a holomorphic surjective submersion to $\mathbb CP^1$ such that the f\/iber over every \smash{$z \in {\mathbb C}P^1$} is identif\/ied with~$(M, I_z)$. The twistor lines are the ``constant'' sections of the above f\/ibration ${\mathcal Z} = M\times\mathbb CP^1 \longrightarrow {\mathbb C} P^1$. These sections are holomorphic and the normal bundle of a~twistor line is isomorphic to ${\mathcal O}_{\mathbb CP^1}(1)^{\oplus d} \longrightarrow \mathbb CP^1$, where $d = \frac{1}{2}\dim_{\mathbb R}M$. This normal bundle is equipped with additional structures which provide all of the information needed to reconstruct the hyper-K\"ahler structure on~$M$; see~\cite{HKLR} for details.

For a complex reductive group $G_{\mathbb C}$, the hyper-K\"ahler structure on the moduli space of f\/lat $G_{\mathbb C}$-connections on a compact Riemann surface $X$ is given by non-abelian Hodge theory. The complex structure $J$ is induced by the complex Lie group $G_{\mathbb C}$ through the identif\/ication of the moduli space with the Betti moduli space $\operatorname{Hom}(\pi_1(X), G_{\mathbb C})/\!\!/G_{\mathbb C}$. Via the solutions of Hitchin's self-duality equation, the moduli space of stable $G_{\mathbb C}$-Higgs bundles, i.e., the Dolbeault moduli space, gets identif\/ied with the moduli space of irreducible f\/lat connections, yielding the complex structure~$I$. The reverse map ${\mathcal M}_{dR}  \longrightarrow \mathcal M_{\rm Dol}$ arises from Donaldson's twisted harmonic maps construction \cite{Do} for the case of $G_{\mathbb C}={\rm SL }(2,\mathbb C)$, see~\cite{Co} for the general case. It was f\/irst shown by Hitchin in the case of ${\rm SL}(2,\mathbb C)$ that~$I$,~$J$ and $K := IJ$ give rise to a~hyper-K\"ahler structure for the natural Riemannian metric. In the abelian case of $G_{\mathbb C} = \mathbb C^*$ the theory simplif\/ies considerably boiling down to the classical abelian Hodge theory~\cite{GX}.

It was f\/irst noticed by Deligne that the twistor space of the moduli space of f\/lat connections on a Riemann surface $X$ has a convenient description by gluing the Hodge moduli spaces for~$X$ and~$\overline X$ via the Riemann--Hilbert isomorphism; see~\cite{Si95,Si08} for details. We recall below this construction for the rank one case (the case of our interests).

Take a point $(\lambda, L, D) \in \mathcal M_{\rm Hod}(X) = \mathcal M^X_{\rm Hod}(1)$ with $\lambda \neq 0$ and consider the f\/lat connec\-tion~$D/\lambda$ on $L$. Let
\begin{gather*}\rho \colon \ \pi_1(X)  \longrightarrow H^1(X, {\mathbb Z}) \longrightarrow \mathbb C^*\end{gather*}
be the monodromy representation for~$D/\lambda$. After identifying $H^1(X, {\mathbb Z})$ with $H^1(\overline{X}, {\mathbb Z})$ by the identity map of the underlying $C^\infty$ manifolds, the homomorphism~$\rho$ gives a~holomorphic rank one $1$-connection $(1, E, \widetilde D)$ on~$\overline{X}$. Consider the morphism~$f$ in~\eqref{lf}. Let
\begin{gather*}
f_{\overline X} \colon \ {\mathcal M}_{\rm Hod}(\overline{X}) := {\mathcal M}^{\overline X}_{\rm Hod}(1)  \longrightarrow {\mathbb C}
\end{gather*}
be the morphism obtained by substituting $\overline X$ in place of~$X$ in \eqref{lf}. Def\/ine a holomorphic map
\begin{gather}\label{hodge-gluing}
\varphi \colon \  f^{-1}({\mathbb C}^*)  \longrightarrow f^{-1}_{\overline X}({\mathbb C}^*) ,\qquad (\lambda, L, D) \longmapsto \big(\lambda^{-1}, E, \widetilde{D}/\lambda\big)
\end{gather}
of the restrictions of the f\/iber bundles ${\mathcal M}_{\rm Hod}(X)$ and ${\mathcal M}_{\rm Hod}(\overline{X})$ to ${\mathbb C}^* \subset \mathbb C$. This $\varphi$ is clearly a~biholomorphism.

\begin{Definition}\label{de1} The rank one Deligne--Hitchin moduli space \begin{gather*}{\mathcal M}_{\rm DH} = {\mathcal M}_{\rm DH}(X)\end{gather*} is
\begin{gather*}{\mathcal M}_{\rm DH}(X) = \mathcal M_{\rm Hod}(X)\cup_\varphi\mathcal M_{\rm Hod}(\overline{X}).\end{gather*}
\end{Definition}

We note that the map $f\vert_{f^{-1}({\mathbb C}^*)}$ coincides with the composition
\begin{gather*}
f^{-1}({\mathbb C}^*) \stackrel{\varphi}{\longrightarrow} f^{-1}_{\overline X}({\mathbb C}^*) \stackrel{f_{\overline X}}{\longrightarrow}
 {\mathbb C}^* \stackrel{z\mapsto z^{-1}}{\longrightarrow} {\mathbb C}^* .
\end{gather*}
Therefore, $f$ and the composition
\begin{gather*}
\mathcal M_{\rm Hod}(\overline{X}) \stackrel{f_{\overline X}}{\longrightarrow}  {\mathbb C} \stackrel{z\mapsto z^{-1}}{\longrightarrow}
{\mathbb C}P^1\setminus\{0\}
\end{gather*}
patch together to produce a morphism
\begin{gather}\label{mapf}
f \colon \ {\mathcal M}_{\rm DH} \longrightarrow \mathbb CP^1 .
\end{gather}
For any $\lambda \in \mathbb CP^1$, the f\/iber $f^{-1}(\lambda)$ will be denoted by $f^\lambda$. The f\/iber $f^{\infty}$ is the moduli space of Higgs line bundles of degree zero on~$\overline X$. This $f$ is the twistor f\/ibration mentioned earlier.

\begin{Definition} The {\it degree} of a smooth map $i \colon \Sigma \longrightarrow {\mathcal M}_{\rm DH}(X)$ from a compact oriented surface $\Sigma$ is the degree of the composition $f\circ i \colon \Sigma \longrightarrow {\mathbb C}P^1$, where $f$ is the projection in~\eqref{mapf}.
\end{Definition}

\begin{Lemma}\label{lemdeg} Let $\Sigma$ be a compact Riemann surface of genus $g_\Sigma$ and $i \colon \Sigma \longrightarrow {\mathcal M}_{\rm DH}(X)$ be a~holomorphic immersion. Then the degree of the corresponding holomorphic normal bundle \begin{gather*}{\mathcal N} = (i^*T{\mathcal
M}_{\rm DH})/T\Sigma\end{gather*} is
\begin{gather*}
\deg ({\mathcal N}) = (2 g_X +2)\deg (i)-\deg (T\Sigma) = (2 g_X +2)\deg (i) +2g_\Sigma-2 ,\end{gather*}
where $g_X$ is the genus of $X$.
\end{Lemma}

\begin{proof} As in the computation in \cite[pp.~555--556]{HKLR} it can be shown that the tangent bundle of the complex manifold ${\mathcal M}_{\rm DH}$ is holomorphically isomorphic to the pull-back
\begin{gather*}f^*\big({\mathcal O}_{{\mathbb C}P^1}(2)\oplus \big({\mathcal O}_{{\mathbb C}P^1}(1)^{\oplus 2g_X}\big)\big) ,\end{gather*}
where the f\/irst summand is the tangent bundle of $\mathbb CP^1$ and the direct sum ${\mathcal O}_{{\mathbb C}P^1}(1)^{\oplus 2g_X}$ corresponds to the hyper-K\"ahler structure on $\mathcal M_{dR}$. Now the lemma follows immediately from the fact that $\text{degree}(i^*f^*{\mathcal O}_{{\mathbb
C}P^1}(1)) = \text{degree}(i)$.
\end{proof}

We will compute $\operatorname{Aut}(\mathcal M_{\rm DH})_0$ by proving a series of lemmas.
\begin{Lemma}\label{sections}
Let $\alpha, \beta, \omega, \eta \in H^0(X, K_X)$ be holomorphic $1$-forms. There exists a holomorphic section $s \colon {\mathbb C}P^1 \longrightarrow {\mathcal M}_{\rm DH}$ defined by
\begin{gather*}s(\lambda) = [\lambda, \overline{\partial}+\overline{\omega}+\lambda\overline{\eta}, \lambda(\partial+\alpha)+\beta] \in \mathcal M_{\rm Hod}(X)
 \subset {\mathcal M}_{\rm DH}\end{gather*}
for $\lambda \in {\mathbb C} \subset \mathbb CP^1$. Conversely, every section of $f \colon {\mathcal M}_{\rm DH} \longrightarrow {\mathbb C}P^1 $ is of this form.
\end{Lemma}

\begin{proof}
The f\/irst part of the lemma is easily verif\/ied. In order to prove the converse direction we start with a lift of the family of holomorphic structures on ${\mathbb C} \subset {\mathbb C}P^1$: As $\mathbb C$ is simply connected, there exists a~map
\begin{gather*}
\lambda \longmapsto \overline{\omega}_1(\lambda) \in \overline{H^0(X, K_X)} , \qquad \lambda \in \mathbb C\end{gather*}
such that \begin{gather*}\phi\circ s(\lambda) = [\overline{\partial}+\overline{\omega}_1(\lambda)]\end{gather*}
for all $\lambda \in {\mathbb C}$, where~$\phi$ is the projection in \eqref{phi} and $[-]$ denotes gauge equivalence class. Let $\beta \in H^0(X, K_X)$ be the (well-def\/ined) Higgs f\/ield of $s(0)$. A lift of $s$ is then given by
\begin{gather*}\widehat{s} (\lambda) = (\lambda,\overline{\partial}+\overline{\omega}_1(\lambda), \beta+\lambda(\partial+\alpha_1(\lambda))\end{gather*}
for some holomorphic map \begin{gather*}\alpha_1 \colon \ \mathbb C \longrightarrow \Omega^{(1,0)}(X) .\end{gather*} The image of $\alpha_1$ is contained in $H^0(X, K_X) \subset \Omega^{(1,0)}(X)$ because the $\lambda$-connections are integrable. Thus, $\lambda \longmapsto s(\lambda)$, $\lambda \in \mathbb C^*$, produces the following family of f\/lat connections on~$X$:
\begin{gather*}\widehat{\nabla}^\lambda = d+\lambda^{-1}\beta+\overline{\omega}_1(\lambda)+ \alpha_1(\lambda) .\end{gather*}

We can repeat this over $\mathbb CP^1\setminus\{0\}$ for the Riemann surface $\overline X$. Indeed, there exists a lift $\widetilde s$ of the form
\begin{gather*}\widetilde{s}(\mu) = (\mu,\partial +\alpha_2(\mu), \mu(\overline{\partial}+ \overline{\omega}_2(\mu))+\overline{\eta})\end{gather*}
with $\mu = \frac{1}{\lambda}$. The corresponding f\/lat connections are denoted by $\widetilde{\nabla}^\lambda$. By the gluing condition in Def\/inition \ref{de1} there exists a~gauge transformation $g(\lambda) \colon X \longrightarrow \mathbb C^*$ for every $\lambda \in \mathbb C^*$ such that \begin{gather*}\widehat{\nabla}^\lambda\cdot g(\lambda) = \widetilde{\nabla}^\lambda ;\end{gather*}
this gauge transformation is unique up to a constant scalar. Since the connection 1-forms of $\widehat{\nabla}^\lambda$ and $\widetilde{\nabla}^\lambda$ are both harmonic 1-forms, their dif\/ference
\begin{gather*}\widehat{\nabla}^\lambda-\widetilde{\nabla}^\lambda = \chi(\lambda)\end{gather*}
is a lattice point
\begin{gather}\label{Gamma}
\chi(\lambda) \in \Lambda := \left\{\!\chi \in \Omega^1(X, \mathbb C) \,|\, d\chi = d*\chi = 0;
\int_\gamma\chi\in2\pi\sqrt{-1}\mathbb Z \text{ for all closed curves } \gamma\!\right\}\!.\!\!\!\!\end{gather}
As the connection 1-forms depend continuously on~$\lambda$, it follows that $\chi$ is independent of $\lambda$. By construction, the connection 1-form of $\lambda \longmapsto \widetilde{\nabla}^\lambda$ has a f\/irst order pole at $\lambda = \infty$ whose residue is~$\overline{\eta}$.
This implies that $\overline{\omega}_1(\lambda)$ must be a polynomial of degree at most one and that $\alpha_1(\lambda)$ is a~constant~$\alpha$.
\end{proof}

\begin{Lemma}\label{tau-lemma} Let $T \colon {\mathcal M}_{\rm DH} \longrightarrow {\mathcal M}_{\rm DH}$ be a~holomorphic automorphism. Then $T$ maps fibers of~$f$ to fibers of~$f$. Moreover, it covers the automorphism $\lambda \longmapsto \tau\lambda$ or $\lambda \longmapsto \frac{\tau}{\lambda}$ of ${\mathbb C}P^1$ for some $\tau \in \mathbb C^*$; if the latter case occurs, then the Jacobian of $X$ is isomorphic to the Jacobian of~$\overline X$.
\end{Lemma}

\begin{proof} We f\/irst prove that two points in the same f\/iber $p_1, p_2 \in f^{\lambda_0} = f^{-1}(\lambda_0)$ are mapped by~$T$ into a~single f\/iber~$f^{\lambda_1}$. To this end, we claim that any two points~$p_1$,~$p_2$ in dif\/ferent f\/ibers can be joined by a holomorphic section ${\mathbb C}P^1 \longrightarrow {\mathcal M}_{\rm DH}$ of the map~$f$. This can be proved quite easily by linear interpolation with sections of the form given in Lemma~\ref{sections}.

If two points $x_1$, $x_2$ of one f\/iber of $f$ are mapped by~$T$ into two dif\/ferent f\/ibers of $f$, we consider a holomorphic section~$s$ of~$f$ passing through~$T(x_1)$ and~$T(x_2)$. Then $\lambda \longmapsto T^{-1}(s(\lambda))$ is a~holomorphic immersion of~${\mathbb C}P^1$ to ${\mathcal M}_{\rm DH}$, and the degree of its normal bundle is the same as the degree of the normal bundle of the section~$s$. From Lemma~\ref{lemdeg} we obtain
\begin{gather*}(2 g_X+2)\deg (i)-2=2g_X,\end{gather*}
which yields $\deg (i)=1$. On the other hand, its degree is at least two because it passes through two points $x_1$, $x_2$ lying in one f\/iber. In view of this contradiction we conclude that~$T$ maps f\/ibers of~$f$ to f\/ibers of~$f$.

Therefore, there is a unique holomorphic map
\begin{gather*}T'\colon \ {\mathbb C}P^1 \longrightarrow {\mathbb C}P^1\end{gather*}
such that $f\circ T = T'\circ f$. As in the proof of Theorem~\ref{thm1}, $T$ must satisfy $T(\{0,\infty\}) = \{0,\infty\}$, and hence $T'$ is of the form $\lambda \longmapsto \tau\lambda$ or $\lambda \longmapsto \frac{\tau}{\lambda}$ for some $\tau \in \mathbb C^*$. If $T(0) = \infty$ the rank one Higgs moduli spaces for $X$ and $\overline X$ are holomorphically isomorphic. Since there is no nonconstant holomorphic map from an abelian variety
to an af\/f\/ine space, any biholomorphism between the rank one Higgs moduli spaces for $X$ and $\overline X$ produces a~biholomorphism between~$\operatorname{Pic}^0(X)$ and~$\operatorname{Pic}^0(\overline{X})$.
\end{proof}

Let $T \colon {\mathcal M}_{\rm DH} \longrightarrow {\mathcal M}_{\rm DH}$ be a holomorphic automorphism that maps each f\/iber of $f$ to itself, meaning $T(f^\lambda) = f^\lambda$ for all $\lambda \in {\mathbb C}P^1$. Consider the projection
\begin{gather}\label{pi2}
\pi_2 \colon \ f^0  = \operatorname{Pic}^0(X)\times H^0(X, K_X) \longrightarrow H^0(X, K_X) .
\end{gather}
Then, the automorphism $T$ restricted to $f^0$ maps f\/ibers of $\pi_2$ to f\/ibers of $\pi_2$ since $\operatorname{Pic}^0(X)$ does not have any nonconstant holomorphic map to $H^0(X, K_X)$. A~corresponding statement is true for the f\/iber $f^\infty$. Thus, the assumptions in the following
lemma are natural.

\begin{Lemma}\label{almost-id} Let $T \colon {\mathcal M}_{\rm DH} \longrightarrow {\mathcal M}_{\rm DH}$ be a~holomorphic automorphism such that
$T(f^\lambda) = f^\lambda$ for all $\lambda \in \mathbb CP^1$. Assume that $T$ restricted to $\operatorname{Pic}^0(X)\times\{0\} \subset f^0$ is the
identity map, and $T$ restricted to $\operatorname{Pic}^0(\overline{X})\times\{0\} \subset f^\infty$ is the identity map. Then, there exists an integral harmonic $1$-form $\gamma \in \Lambda = H^1(X, \mathbb Z) \subset \operatorname{Harm}(X; \mathbb C)$ such that $T$ is given by
\begin{gather*}T((\lambda, E, D)) = (\lambda, E, D+\lambda\gamma') ,\end{gather*}
where $\gamma = \gamma'+\gamma''$ is the decomposition into $(1,0)$ and $(0,1)$ components.
\end{Lemma}

\begin{proof}
Consider the ``constant'' section of $f$
\begin{gather*}s(\lambda) = [\lambda, \overline{\partial}, \lambda \partial],\end{gather*}
where $d = \partial+\overline{\partial}$ is the decomposition into types. From Lemma~\ref{sections} and the assumption that~$T$ restricted to $\operatorname{Pic}^0(X)\times\{0\} \subset f^0$ is the identity map it follows that $T\circ s$ must be of the form
\begin{gather}\label{specialT}
T\circ s(\lambda) = (\lambda, \overline{\partial}+\gamma_1'', \lambda (\partial+\gamma_2'))
\end{gather}
for some $\gamma_1, \gamma_2 \in \Lambda$. By applying the gauge $\exp\big({-}\int\gamma_1\big)$, the identity in \eqref{specialT} is equivalent to
\begin{gather*}T\circ s(\lambda) = (\lambda, \overline{\partial}, \lambda (\partial+\gamma'))\end{gather*}
 for $\gamma = \gamma_2-\gamma_1$.

Since $T$ is continuous it follows that for every constant section of~$f$ def\/ined by $s_D(\lambda) = (\lambda, L, \lambda D)$, where $D$ is a~holomorphic connection on~$L$, the equality
\begin{gather*}T\circ s_D(\lambda) = (\lambda, L, \lambda (D+\gamma'))\end{gather*}
holds. It should be mentioned that this is equivalent to
\begin{gather*}T\circ s(\lambda) = (\lambda, L\otimes L(\overline{\partial}-\gamma''), \lambda D) .\end{gather*}
Note that for any $\lambda \in {\mathbb C}\setminus\{0\}$ and for every point $p = (\lambda, E, D) \in f^\lambda$, there exists a constant section
which goes through $p$. This completes the proof.
\end{proof}

Let \begin{gather*}s \colon \ {\mathbb C}P^1 \longrightarrow {\mathcal M}_{\rm DH} ,\qquad
\lambda \longmapsto [\lambda, \overline{\partial} +\overline{\omega}+ \lambda\overline{\eta}, \lambda(\partial+\alpha)+\beta]
\end{gather*}
be a section of $f$, where $\alpha, \beta, \eta, \omega \in H^0(X, K_X)$; def\/ine
\begin{gather}\label{Ts}
T_s \colon \ {\mathcal M}_{\rm DH} \longrightarrow {\mathcal M}_{\rm DH} ,\qquad [\lambda, E, D] \longmapsto [\lambda, E\otimes
L(\overline{\partial}+\overline{\omega}+\lambda\overline{\eta}), D\otimes(\lambda(\partial+\alpha)+\beta)] ,
\end{gather}
where the tensor product of two $\lambda$-connections~$D$,~$\widetilde{D}$ on two holomorphic line bundles~$E$,~$\widetilde E$ respectively is the
$\lambda$-connection
\begin{gather*}D\otimes\widetilde{D}(e\otimes\widetilde{e}) = D(e)\otimes \widetilde{e}+e\otimes \widetilde{D}(\widetilde {e})\end{gather*}
on $E\otimes\widetilde{E}$. This $T_s$ is clearly a holomorphic automorphism of ${\mathcal M}_{\rm DH}$.

Let
\begin{gather*}
\operatorname{Aut}({\mathcal M}_{\rm DH})_0 \subset \operatorname{Aut}({\mathcal M}_{\rm DH})
\end{gather*}
be the connected component, containing the identity element, of the group of holomorphic automorphisms of ${\mathcal M}_{\rm DH}$. Using Lemmas~\ref{tau-lemma} and~\ref{almost-id} we can give the following description of $\operatorname{Aut}({\mathcal M}_{\rm DH})_0$.

\begin{Theorem}\label{thmDHa}
There is an isomorphism
\begin{gather*}\operatorname{Aut}({\mathcal M}_{\rm DH})_0 =
\mathcal M_{dR}(X)\times \big(\big(H^0(X, K_X)\times\overline{H^0(X, K_X)}\big)\rtimes {\mathbb C}^*\big)\end{gather*}
with the group structure of the right-hand side given by
\begin{gather}\label{gs37}
\big(\nabla^1, \alpha_1, \overline{\eta}_1, \tau_1\big)\cdot \big(\nabla^2, \alpha_2,\overline{\eta}_2, \tau_2\big) = \big(\nabla^1\otimes\nabla^2, \alpha_1+\tau_1\alpha_2, \overline{\eta}_1+\tfrac{1}{\tau_1}\overline{\eta}_2, \tau_1\tau_2\big) .
\end{gather}
The group $\mathcal M_{dR}(X)$ acts via the automorphisms $T_s$ in~\eqref{Ts}, $H^0(X, K_X)\times\overline{H^0(X, K_X)}$ acts using addition of forms to connections and Higgs bundles, and ${\mathbb C}^*$ acts via multiplication.
\end{Theorem}

\begin{proof} Using Lemma \ref{tau-lemma} and the lift of the $\mathbb C^*$ action on $\mathbb C P^1$ we can restrict to automorphisms $T\in \operatorname{Aut}(\mathcal M_{\rm DH})_0$ such that
\begin{gather*}T\big(f^\lambda\big) = f^\lambda\end{gather*} for all $\lambda\in\mathbb CP^1$. We claim that there exists a section $s \colon {\mathbb C}P^1 \longrightarrow {\mathcal M}_{\rm DH}$ of $f$ such that $T_s\circ T$ is of the form given in Lemma~\ref{almost-id}, where $T_s$ is constructed in~\eqref{Ts}.

Recall that $T$ acts f\/iberwise on the f\/ibers of~$\pi_2$ in~\eqref{pi2}. An automorphism of $\operatorname{Pic}^0(X)$ which is homotopic to the identity map is a translation by an element of $\operatorname{Pic}^0(X)$. Hence, there exists $\alpha, \omega \in H^0(X, K_X)$ such that
\begin{gather*}T(0, L, 0) = (0, L\otimes L(\overline{\partial}+\overline{\omega}), \alpha)\end{gather*}
for all $(0, L, 0) \in \pi_2^{-1}(0) \subset f^0$. The same argument yields that there are $\beta, \eta \in H^0(X, K_X)$ such that the automorphism $T_s$, for the section $s$ def\/ined by
\begin{gather*}s(\lambda) = (\lambda, \overline{\partial} -\overline{\omega}-\lambda\overline{\eta}, \lambda(\partial-\beta)-\alpha) ,\end{gather*}
has the desired properties.

Moreover, by applying the arguments in the proof of Lemma~\ref{almost-id}, we see that a~section~$s$ of~$f$ such that $T_s\circ T$ is of the form given in Lemma~\ref{almost-id} is unique up to tensoring with a constant section of the form $\lambda \longmapsto (\lambda, \overline{\partial}, \lambda(\partial+\gamma'))$ for some $\gamma \in H^1(X, \mathbb Z) \subset \operatorname{Harm}(X, \mathbb C)$.

Altogether, we obtain that any element in $\operatorname{Aut}({\mathcal M}_{\rm DH})_0$ is of the form $(\nabla, \alpha, \overline{\eta}, \tau)\in\mathcal M_{dR}(X)\times \big(\big(H^0(X, K_X)\times\overline{H^0(X, K_X)}\big)\rtimes {\mathbb C}^*\big)$ for the actions of the components $\mathcal M_{dR}(X)$, $H^0(X, K_X)$, $\overline{H^0(X, K_X)}$ and ${\mathbb C}^*$ on ${\mathcal M}_{\rm DH}$ as described above. The group structure of $\operatorname{Aut}({\mathcal M}_{\rm DH})_0$ is easily verif\/ied to be as in \eqref{gs37}.
\end{proof}

\subsection[Automorphisms of ${\mathcal M}_{\rm DH}$ preserving a cohomology class]{Automorphisms of $\boldsymbol{{\mathcal M}_{\rm DH}}$ preserving a cohomology class}

In this subsection we assume that $g = \operatorname{genus}(X) > 1$.

Till now we have restricted to automorphism of ${\mathcal M}_{\rm DH}$ lying in ${\rm Aut}({\mathcal M}_{\rm DH})_0$. In this subsection we def\/ine a natural cohomology class on ${\mathcal M}_{\rm DH}$ and compute all automorphisms which preserve this particular class.

Any holomorphic automorphism $t \colon X \longrightarrow X$ induces an automorphism $T_t \colon{\mathcal M}_{\rm DH} \longrightarrow {\mathcal M}_{\rm DH}$ via pull-back; note that $t$ is also a holomorphic automorphism of~$\overline X$. More precisely, for $\lambda \in \mathbb C^*$ and $(\lambda, L, D) \in f^\lambda$, def\/ine
\begin{gather}\label{Teq}
T_t(\lambda, L, D) = (\lambda, t^*L, t^* D) .
\end{gather}
It is straight-forward to check that \eqref{Teq} extends holomorphically to $f^0$ and $f^\infty$ giving rise to a~holomorphic automorphism of ${\mathcal M}_{\rm DH}$; the inverse of $T_t$ is $T_{t^{-1}}$. Similarly, if \begin{gather*}s \colon \ \overline{X} \longrightarrow X\end{gather*} is a~holomorphic isomorphism, then def\/ine
\begin{gather*}
T_{s} \colon \ \mathcal M_{\rm Hod}(X) \longrightarrow \mathcal M_{\rm Hod}(\overline{X}) ,\qquad (\lambda, L, D) \longmapsto (\lambda, s^*L, s^*D).
\end{gather*}
But $\mathcal M_{\rm Hod}(\overline{X})$ is identif\/ied with ${\mathcal M}_{\rm DH}\setminus f^0$ (this follows immediately from~\eqref{hodge-gluing}). Therefore, $T_{s}$ gives a biholomorphism of $\mathcal M_{\rm Hod}(X)$ with ${\mathcal M}_{\rm DH}\setminus f^0$. This biholomorphism extends~$f^\infty$, and hence~$T_{s}$ is a holomorphic automorphism of ${\mathcal M}_{\rm DH}$ over the involution $\lambda \longmapsto \frac{1}{\lambda}$ of $\mathbb CP^1$. Explicitly, for $\lambda\neq0,\infty$ it is given by
\begin{gather*}T_{s}(\lambda, L, D)=\varphi^{-1}((\lambda, s^*L, s^* D)).\end{gather*}

These examples can be generalized as follows. Consider $\Lambda$ in~\eqref{Gamma}; let $\Lambda''$ be its image in $\overline{H^0(X, K_X)} \subset \operatorname{Harm}^1(X; {\mathbb C})$. Then
\begin{gather*}\operatorname{Pic}^0(X) = \overline{H^0(X, K_X)}/\Lambda'' ,\end{gather*} and every $\gamma'' \in \Lambda''$ determines a unique $\gamma \in \Lambda$ such that $\gamma = (\gamma-\gamma'')+\gamma''$ is the decomposition into $(1,0)$ and $(0,1)$ parts. As $\gamma \in \Lambda$ is imaginary we get that $\gamma' = (\gamma-\gamma'') = -\overline{\gamma}''$. Hence, any holomorphic automorphism
\begin{gather*}
t \colon  \ \operatorname{Pic}^0(X) \longrightarrow \operatorname{Pic}^0(X)
\end{gather*}
with $t(0) = 0$ gives rise to a map
\begin{gather*}\widehat{t} = \overline{t}\oplus t \colon \ H^0(X, K_X)\oplus \overline{H^0(X, K_X)} \longrightarrow H^0(X, K_X)\oplus
\overline{H^0(X, K_X)}\end{gather*}
satisfying
\begin{gather*}\widehat{t}(\Lambda) = \Lambda .\end{gather*}
Explicitly, $\widehat t$ is given by
\begin{gather*}\widehat t(\alpha\oplus\overline{\eta}) = \overline{ t(\overline{\alpha})}\oplus
t(\overline{\eta})
\end{gather*}
for $\alpha,\eta\in H^0(X, K_X)$. Decompose $d = \partial+\overline{\partial}$. For any $\lambda \in \mathbb C^*$, every element $p \in f^\lambda$ is of the form
\begin{gather*}p = (\lambda, \overline{\partial}+\overline{\omega}, \lambda(\partial+\alpha))\end{gather*}
for some $\alpha, \omega \in H^0(X, K_X) $. We then def\/ine
\begin{gather}\label{Teq2}
T(\lambda, \overline{\partial}+\overline{\omega}, \lambda(\partial+\alpha)) = (\lambda, \overline{\partial}+t(\overline{\omega}),  \lambda(\partial+\overline{t}(\alpha))).
\end{gather}
Again, \eqref{Teq2} yields a well-def\/ined automorphism of ${\mathcal M}_{\rm DH}$ which covers the identity on ${\mathbb C}P^1$. Finally if $s \colon \operatorname{Pic}^0(X) \longrightarrow \operatorname{Pic}^0(\overline{X})$ is an isomorphism satisfying $s(0) = 0 \in \operatorname{Pic}^0(\overline{X})$, we can construct an of\/f-diagonal isomorphism
\begin{gather*}\widehat{s}=\left( \begin{matrix}
 0 & s \\  \overline{s} & 0   \end{matrix}  \right) \colon \ H^0(X, K_X)\oplus\overline{H^0(X, K_X)}
 \longrightarrow H^0(X, K_X)\oplus\overline{H^0(X, K_X)}\end{gather*}
with
\begin{gather*}\widehat{s}(\Lambda) = \Lambda .\end{gather*}
Thus, the formula
\begin{gather}\label{barTeq2}
T(\lambda, \overline{\partial}+\overline{\omega}, \lambda(\partial+\alpha)) = (\lambda, \partial+s(\overline{\omega}), \lambda(\overline{\partial}+
\overline{s}(\alpha))) \in {\mathcal M}_{\rm DH}(\overline{X})  = {\mathcal M}_{\rm DH}( X)
\end{gather}
def\/ines an automorphism of ${\mathcal M}_{\rm DH}( X)$ covering $\lambda \longmapsto \lambda^{-1}$.

\begin{Proposition}\label{prop3}
The components of the space of automorphisms of ${\mathcal M}_{\rm DH}$ are given by the space of holomorphic isomorphisms $t \colon \operatorname{Pic}^0(X) \longrightarrow \operatorname{Pic}^0(X)$ with $t(0) = 0$ together with the space of holomorphic isomorphisms $s \colon \operatorname{Pic}^0(X) \longrightarrow \operatorname{Pic}^0(\overline{X})$ with $s(0) = 0$.
\end{Proposition}

\begin{proof} Let $T \colon {\mathcal M}_{\rm DH} \longrightarrow {\mathcal M}_{\rm DH}$ be a holomorphic automorphism which maps the f\/iber $f^0$ to itself. After composing with a suitable element of $\operatorname{Aut}({\mathcal M}_{\rm DH})_0$, compare with Section~\ref{aut0mdh} and in particular with the proof of Theorem~\ref{thmDHa}, we may assume that $T$ maps $\{0\}\times\operatorname{Pic}^0(X)\times\{0\} \subset f^0$ to itself and $\{\infty\}\times\operatorname{Pic}^0(\overline{X}) \times\{0\} \subset f^\infty$ to itself such that $T(0, \overline{\partial}, 0) = (0, \overline{\partial}, 0)$ and $T(\infty, \partial, 0) = (\infty, \partial, 0)$ (where $d = \partial+\overline{\partial}$ as before). We need to show that $T$ is, up to another automorphism in the identity component, of the form~\eqref{Teq2}. In view of Lemma~\ref{almost-id} it remains to prove that the restrictions $t_1 \colon \operatorname{Pic}^0(X) \longrightarrow \operatorname{Pic}^0(X)$ and $t_2 \colon \operatorname{Pic}^0(\overline{X}) \longrightarrow \operatorname{Pic}^0(\overline{X})$ f\/it together, i.e.,
\begin{gather}\label{whitehatt}\widehat{t}_1 = t_2\oplus t_1\end{gather}
on $H^0(X, K_X)\oplus \overline{H^0(X, K_X)}$. This follows from an argument similar to one in the proof of Lemma~\ref{almost-id}: First note that constant sections
\begin{gather*}\lambda \longmapsto (\lambda, \overline{\partial}+\overline{\omega},
\lambda(\partial+\alpha))\end{gather*}
are mapped to constant sections of the form
\begin{gather*}\lambda \longmapsto (\lambda, \overline{\partial}+t_1(\overline{\omega}), \lambda(\partial+t_2(\alpha)))\end{gather*}
for af\/f\/ine linear lifts $t_1 \colon \overline{H^0(X, K_X)}  \longrightarrow \overline{H^0(X, K_X)}$ and $t_2 \colon H^0(X, K_X) \longrightarrow H^0(X, K_X)$ of $t_1 \colon \operatorname{Pic}^0(X) \longrightarrow \operatorname{Pic}^0(X)$ and $t_2 \colon \operatorname{Pic}^0(\overline{X}) \longrightarrow \operatorname{Pic}^0(\overline{X})$. Actually, after applying automorphisms of the form given in Lemma \ref{almost-id}, we can assume that both lifts are linear. If we apply the above considerations to $\overline{\omega}=\gamma''$ and $\alpha=\gamma'$ where $\gamma=\gamma'+\gamma''\in\Lambda$ we get that $t_1(\gamma'')\in\Lambda''$, $t_2(\gamma')\in\Lambda'$ and
 \begin{gather*}t_2(\gamma')+t_1(\gamma'')\in\Lambda\end{gather*}
as \looseness=1 we have the constant section which goes through the trivial gauge class at $\lambda=1$. Hence, \eqref{whitehatt}~holds and if $\operatorname{Pic}^0(X)$ and $\operatorname{Pic}^0(\overline{X})$ are not isomorphic the proposition follows from Lem\-ma~\ref{tau-lemma}.

If $\operatorname{Pic}^0(X)$ and $\operatorname{Pic}^0(\overline{X})$ are isomorphic, the proposition follows from the f\/irst part of the proof together with the observation that \eqref{barTeq2} gives an automorphism of ${\mathcal M}_{\rm DH}$.
\end{proof}

\begin{Remark}\label{re-1} We note that $\operatorname{Aut}({\mathcal M}_{\rm DH})_0$ is a normal subgroup of~$\operatorname{Aut}({\mathcal M}_{\rm DH})$, and the quotient group $\operatorname{Aut}({\mathcal M}_{\rm DH})/\operatorname{Aut}({\mathcal M}_{\rm DH})_0$ is identif\/ied with the connected components of $\operatorname{Aut}({\mathcal M}_{\rm DH})$. In particular, the connected components of $\operatorname{Aut}({\mathcal M}_{\rm DH})$ is also a group. Proposition~\ref{prop3} describes the connected components, while Theorem~\ref{thmDHa} describes $\operatorname{Aut}({\mathcal M}_{\rm DH})_0$. However, these two results do not determine the group structure of $\operatorname{Aut}({\mathcal M}_{\rm DH})$.
\end{Remark}

The complex manifold ${\mathcal M}_{\rm DH} = {\mathcal M}_{\rm DH}(X)$ has a natural integer cohomology class $\theta_{X} \in H^2({\mathcal M}_{\rm DH}, {\mathbb Z})$ which can be constructed as follows: Consider the map~$\phi$ in~\eqref{phi}. Identify $\operatorname{Pic}^0(X)$ with the character variety $\operatorname{Hom}(\pi_1(X), \text{U}(1))$ as before, and consider $\phi$ as a map to $\operatorname{Hom}(\pi_1(X), \text{U}(1))$. Let
\begin{gather*}
\phi_{\overline X} \colon \ {\mathcal M}_{\rm Hod}({\overline X}) \longrightarrow \operatorname{Hom}(\pi_1(X), \text{U}(1)) = \operatorname{Hom}(\pi_1({\overline X}), \text{U}(1))
\end{gather*}
be the map obtained by substituting ${\overline X}$ in place of $X$ in the construction of $\phi$. The two maps $\phi$ and $\phi_{\overline X}$ coincide over the intersection ${\mathcal M}_{\rm Hod}(X)\cap \mathcal M_{\rm Hod}(\overline{X}) = f^{-1}({\mathbb C}^*) \subset {\mathcal M}_{\rm DH}(X)$. Consequently, we get a map
\begin{gather*}
\phi_{{\rm DH},X} \colon \  {\mathcal M}_{\rm DH}(X) \longrightarrow \operatorname{Hom}(\pi_1(X), \text{U}(1)) .
\end{gather*}
Now def\/ine
\begin{gather*}
\theta_{X} := \phi^*_{{\rm DH},X}\theta_0 \in H^2({\mathcal M}_{\rm DH}(X), {\mathbb Z}),
\end{gather*}
where $\theta_0$ is the integral cohomology class in \eqref{the}.

The symplectic form $\widetilde\theta$ on the $\operatorname{Hom}(\pi_1(X), \text{U}(1))$ (see~\eqref{the}) depends on the orientation of~$X$, i.e., if the orientation of the topological surface underlying $X$ is changed, then the corresponding symplectic form is $-\widetilde\theta$. Therefore,
with respect to the isomorphism ${\mathcal M}_{\rm DH}(X) \stackrel{\sim}{\longrightarrow} {\mathcal M}_{\rm DH}(\overline{X})$, we have
\begin{gather*}
\theta_{X} = -\theta_{\overline X} ,
\end{gather*}
where $\theta_{\overline X} \in H^2({\mathcal M}_{\rm DH}({\overline X}), {\mathbb Z})$ is obtained by substituting ${\overline X}$ in place of $X$ in the construction of~$\theta_{X}$. Let
\begin{gather}\label{atx}
\operatorname{Aut}({\mathcal M}_{\rm DH}(X))^{\theta_X} \subset \operatorname{Aut}({\mathcal M}_{\rm DH}(X))
\end{gather}
be the subgroup that f\/ixes the cohomology class $\theta_X$.

For any holomorphic automorphism $t \in \operatorname{Aut}(\operatorname{Pic}^0(X))$, the corresponding automorphism $T_t$ of ${\mathcal M}_{\rm DH}(X)$ f\/ixes the cohomology class $\theta_X$ if and only if \begin{gather*}t \in \Gamma_X\end{gather*} (def\/ined in~\eqref{GX}); see the proof of Proposition~\ref{propo1}.

As in Section \ref{Autocoh} we have an action of $\Gamma_X$ on
$\operatorname{Aut}({\mathcal M}_{\rm DH})$ which preserves the
identity component $\operatorname{Aut}({\mathcal M}_{\rm DH})_0 \subset \operatorname{Aut}({\mathcal
M}_{\rm DH})$. Consider the corresponding semi-direct product
\begin{gather*}
\widetilde{\Gamma}_{{\rm DH},X} :=
\operatorname{Aut}({\mathcal M}_{\rm DH})_0\rtimes \Gamma_X .
\end{gather*}
Then, we obtain the following:

\begin{Theorem}The natural homomorphism \begin{gather*}\widetilde{\Gamma}_{{\rm DH},X} \longrightarrow
\operatorname{Aut}({\mathcal M}_{\rm DH}(X))^{\theta_X} ,\end{gather*} where $\operatorname{Aut}({\mathcal
M}_{\rm DH}(X))^{\theta_X}$ is defined in \eqref{atx}, is an isomorphism.
\end{Theorem}

\begin{proof}
Any holomorphic automorphism of ${\mathcal M}_{\rm DH}(X)$ takes the union
$f^0\cup f^\infty$ to itself. Indeed, this follows from the following two facts:
\begin{itemize}\itemsep=0pt
\item all compact complex submanifolds of ${\mathcal M}_{\rm DH}(X)$ of dimension at least two are contained in $f^0\cup f^\infty$, and
\item the union $f^0\cup f^\infty$ is covered by compact complex submanifolds of dimension $g \geq 2$.
\end{itemize}

We next observe that any automorphism $\xi \in \operatorname{Aut}({\mathcal M}_{\rm DH}(X))^{\theta_X}$ preserves $f^0$ and $f^\infty$ indivi\-dual\-ly. To see this, let
\begin{gather*}\xi' \colon \ f^0 \longrightarrow f^\infty\end{gather*} be a biholomorphism. Since the maximal dimensional compact complex submanifolds of $f^0$ (respectively, $f^\infty$ are isomorphic to $\operatorname{Pic}^0(X)$ (respectively, $\operatorname{Pic}^0(\overline{X})$), we conclude that $\xi'$ produces a holomorphic isomorphism of $\operatorname{Pic}^0(X)$ with $\operatorname{Pic}^0(\overline{X})$. But the cohomology class of $\widetilde\theta$ in \eqref{tt} gives a negative class on $\operatorname{Pic}^0(\overline{X})$ after identifying $\operatorname{Pic}^0(\overline{X})$ with $\operatorname{Hom}(\pi_1(X), \text{U}(1)) = \operatorname{Hom}(H_1(X, {\mathbb Z}), \text{U}(1))$. On the other hand, the class $\theta_0 \in H^2(\operatorname{Pic}^0(X), {\mathbb Z})$ given by $\widetilde\theta$ is ample. A holomorphic isomorphism between complex projective manifolds can't take an ample class to a~negative class. Hence any automorphism $\xi \in \operatorname{Aut}({\mathcal M}_{\rm DH}(X))^{\theta_X}$ takes~$f^0$ (respectively,~$f^\infty$) to~$f^0$ (respectively,~$f^\infty$).

Now the proof is analogous to the proof of Proposition \ref{propo1}.
\end{proof}

\section{Algebraic dimension of the Deligne--Hitchin moduli space}

\begin{Lemma}\label{lem-m} Consider the constant ``trivial'' section
\begin{gather*}
s \colon \ {\mathbb C}P^1 \longrightarrow {\mathcal M}_{\rm DH}(X) ,\qquad \lambda \longmapsto (\lambda, \overline{\partial}, \lambda\partial)
\end{gather*}
of $f$ $($see \eqref{mapf}$)$. There exists an analytic open neighborhood ${\mathcal U} \subset {\mathcal M}_{\rm DH}(X)$ containing $s\big({\mathbb C}P^1\big)$ and an analytic open neighborhood $\mathcal V$ of the zero section of
\begin{gather*}{\mathcal O}_{{\mathbb C}P^1}(1)^{\oplus 2g} \longrightarrow {\mathbb C}P^1\end{gather*}
such that $\mathcal U$ and $\mathcal V$ are biholomorphic.
\end{Lemma}

\begin{proof} Identify ${\mathbb C}P^1$ with lines in ${\mathbb C}^2$ by sending any $\lambda \in {\mathbb C} \subset {\mathbb C}P^1$ to the line generated by $(\lambda, 1) \in {\mathbb C}^2$ while the point $\infty = {\mathbb C}P^1 \setminus {\mathbb C}$ is sent to the line in ${\mathbb C}^2$ generated by $(1, 0)$. The f\/iber of the line bundle ${\mathcal O}_{{\mathbb C}P^1}(1)$ over any point $z \in {\mathbb C}P^1$ is the dual line $(\ell^z)^*$, where $\ell^z \subset {\mathbb C}^2$ is the line corresponding to $z$. Hence ${\mathcal O}_{{\mathbb C}P^1}(1)$ has a canonical trivialization over $\mathbb C$ (respectively, $\mathbb CP^1\setminus\{0\}$) given by $\lambda \longmapsto \{c\cdot (\lambda, 1) \longmapsto c\}$ (respectively, $\lambda \longmapsto \{c\cdot (1, 1/\lambda)  \longmapsto c\}$).

Let $\omega_1, \dots ,\omega_g$ be a basis of $H^0(X, K_X)$, where $g = \operatorname{genus}(X)$, and let $\overline{\omega}_1, \dots , \overline{\omega}_g$ be the corresponding conjugate basis of $\overline{H^0(X, K_X)}$. Now let
\begin{gather*}
\Phi \colon \ {\mathbb C}\times {\mathbb C}^{2g} \longrightarrow {\mathcal M}_{\rm DH}(X),\\
\hphantom{\Phi \colon}{} \ (\lambda, \chi_1, \dots ,\chi_g, \alpha_1, \dots , \alpha_g) \longmapsto
\left(\lambda, \overline{\partial}-\sum_{i=1}^g\chi_i\overline{\omega}_i, \lambda \partial+\sum_{i=1}^g\alpha_i\omega_i\right)
\end{gather*}
be the map. This map $\Phi$ extends to a holomorphic map
\begin{gather*}
{\mathbb C}\times {\mathbb C}^{2g} \subset {\mathcal O}_{{\mathbb C}P^1}(1)^{\oplus 2g} \stackrel{\widetilde\Phi}{\longrightarrow}
{\mathcal M}_{\rm DH}(X) ;
\end{gather*}
recall that the trivial line bundle ${\mathbb C}\times {\mathbb C} \longrightarrow {\mathbb C}$ is identif\/ied with the restriction of ${\mathcal O}_{{\mathbb C}P^1}(1)$ to ${\mathbb C} \subset {\mathbb C}P^1$. This holomorphic map $\widetilde\Phi$ is a biholomorphism from an analytic neighborhood of the zero section of ${\mathcal O}_{{\mathbb C}P^1}(1)^{\oplus 2g} \longrightarrow {\mathbb C}P^1$ to a neighborhood of $s({\mathbb C}P^1)$.
\end{proof}

The algebraic dimension of a compact connected complex manifold $Y$ is the transcendence degree of the f\/ield of global meromorphic functions on $Y$ over the f\/ield $\mathbb C$. The algebraic dimension is bounded above by the complex dimension. If the algebraic dimension of $Y$ coincides with the complex dimension of $Y$, then $Y$ is called Moishezon.

The def\/inition of algebraic dimension does not make sense in general if $Y$ is not compact. For example the transcendence degree, over the f\/ield $\mathbb C$, of the f\/ield of global meromorphic functions on the open unit disk is inf\/inite. However, a theorem of Verbitsky says that the notion of algebraic dimension continues to remain valid, and it is bounded above by the complex dimension, if $Y$ contains an ample rational curve \cite[p.~329, Theorem~1.4]{Ve}. In particular, the notion of algebraic dimension extends to the twistor spaces.

\begin{Proposition}\label{prop-m}
The algebraic dimension of ${\mathcal M}_{\rm DH}(X)$ coincides with $\dim {\mathcal M}_{\rm DH}(X)$, in other words, ${\mathcal M}_{\rm DH}(X)$ is Moishezon.
\end{Proposition}

\begin{proof} Consider the compactif\/ication $P\big({\mathcal O}_{{\mathbb C}P^1}(1)^{\oplus 2g}\oplus{\mathcal O}_{{\mathbb C}P^1}\big)$ of ${\mathcal O}_{{\mathbb C}P^1}(1)^{\oplus 2g}$. Since it is a~complex projective manifold, it is Moishezon. Now from \cite[p.~337, Theorem~3.4]{Ve} we know that any analytic neighborhood of the zero section of ${\mathcal O}_{{\mathbb C}P^1}(1)^{\oplus 2g}$ is Moishezon, because the zero section is an ample rational curve. Therefore, from Lemma \ref{lem-m} and \cite[p.~337, Theorem~3.4]{Ve} it follows that ${\mathcal M}_{\rm DH}(X)$ is Moishezon.
\end{proof}

\subsection*{Acknowledgements}

We thank the referees for their detailed and helpful comments. The work begun during a~research stay of the second author at the Tata Institute of Fundamental Research and he would like to thank the institute for its hospitality. SH is partially supported by DFG HE 6818/1-2. The f\/irst author is partially supported by a~J.C.~Bose Fellowship.

\pdfbookmark[1]{References}{ref}
\LastPageEnding

\end{document}